\documentclass[11pt]{amsart}
\usepackage{amsmath, amsfonts, amsthm, amssymb, multicol, mathtools, dsfont, verbatim, enumerate}
\usepackage{graphicx}
\usepackage{float, hyperref}
\usepackage{scalerel,stackengine, subcaption}
\usepackage[usenames,dvipsnames]{xcolor}
\usepackage{enumitem}
\usepackage{pgfplots}
\usepgfplotslibrary{fillbetween}
\pgfplotsset{width=10cm,compat=1.9}
\usepackage[square,sort,comma,numbers]{natbib}
\allowdisplaybreaks

\usepackage{fancyhdr}
 
\numberwithin{equation}{section}

\newtheorem{thm}{Theorem}[section]
\newtheorem{lma}[thm]{Lemma}
\newtheorem{cor}[thm]{Corollary}

\newtheorem{ques}[thm]{Question}

\renewcommand{\epsilon}{\varepsilon}

\newcommand{\rd}{\mathbb{R}^d}

\newcommand{\conp}{\text{Con}^+(n)}

\renewcommand{\geq}{\geqslant}
\renewcommand{\leq}{\leqslant}
\renewcommand{\emptyset}{\varnothing}
\renewcommand{\epsilon}{\varepsilon}

\renewcommand{\geq}{\geqslant}
\renewcommand{\leq}{\leqslant}

\newcommand{\ubd}{\overline{\dim}_{\textup{B}}}

\newcommand{\lbd}{\underline{\dim}_{\textup{B}}}
\newcommand{\hd}{\dim_{\textup{H}}}

\newcommand{\bd}{\dim_{\textup{B}}}

\newcommand{\I}{\textbf{i}}

\makeatletter
\def\@setauthors{%
  \begingroup
  \def\thanks{\protect\thanks@warning}%
  \trivlist
  \centering\footnotesize \@topsep30\p@\relax
  \advance\@topsep by -\baselineskip
  \item\relax
  \author@andify\authors
  \def\\{\protect\linebreak}

  \normalsize\lowercase{\authors}%
  
	\ifx\@empty\contribs
  \else
    ,\penalty-3 \space \@setcontribs
    \@closetoccontribs
  \fi
  \endtrivlist
  \endgroup
}
\def\@settitle{\begin{center}
\LARGE\lowercase{\@title}
  \end{center}%
}
\makeatother
\newcommand{\authoremail}[1]{\email{\href{mailto:#1}{\color{lightblue}{#1}}}}
\newcommand{\authoraddress}[1]{\address{\normalfont{#1}}}

\definecolor{lightblue}{HTML}{2B77A4}
\definecolor{darkred}{HTML}{9E0D0D}
\hypersetup{
	colorlinks=true,
	linkcolor=darkred,
	urlcolor=darkred,
	citecolor=lightblue
}
\title{Inhomogeneous attractors and  box dimension}

\author{Jonathan M. Fraser}
\authoraddress{Jonathan M. Fraser, University of St Andrews, Scotland}
\authoremail{jmf32@st-andrews.ac.uk}
\thanks{JMF was financially supported by  a  \emph{Leverhulme Trust Research Project Grant} (RPG-2019-034) and an \emph{EPSRC Standard Grant} (EP/Y029550/1).}

\date{}

\begin{document}


\maketitle
\thispagestyle{empty}

\begin{abstract}
Iterated function systems (IFSs) are one of the most important tools for building examples of fractal sets exhibiting some kind of `approximate  self-similarity'.  Examples include self-similar sets, self-affine sets etc.  A beautiful variant on the standard IFS model was introduced by Barnsley and Demko in 1985 where  one builds an \emph{inhomogeneous} attractor by taking the closure of the orbit of a fixed compact condensation set under a given standard IFS.  In this expository  article I will discuss the dimension theory of inhomogeneous attractors, giving several examples and some open questions.   I will focus on the upper box dimension with emphasis on  how to derive good estimates, and when these estimates fail to be sharp.
\end{abstract}

\tableofcontents


\section{Introduction}

\subsection{Inhomogeneous attractors}

Let $X$ denote a compact subset of Euclidean space (often the closed unit cube or unit ball). An iterated function system (IFS) is a finite collection $\mathbb{I}=\{ S_i \}_{i=1}^{N}$ of contraction mappings which map $X$ into itself.    It is a fundamental result in fractal geometry due to Hutchinson (see \cite{falconer, hutchinson}) that for every IFS there exists a unique non-empty compact set,  $F$, called the attractor, which satisfies
\begin{equation} \label{hom}
F = \bigcup_{i=1}^{N} S_i (F).
\end{equation}
We call such attractors \emph{homogeneous} attractors.

\begin{figure}[H] \label{examplesss}
	\centering
	\includegraphics[width=0.48\textwidth]{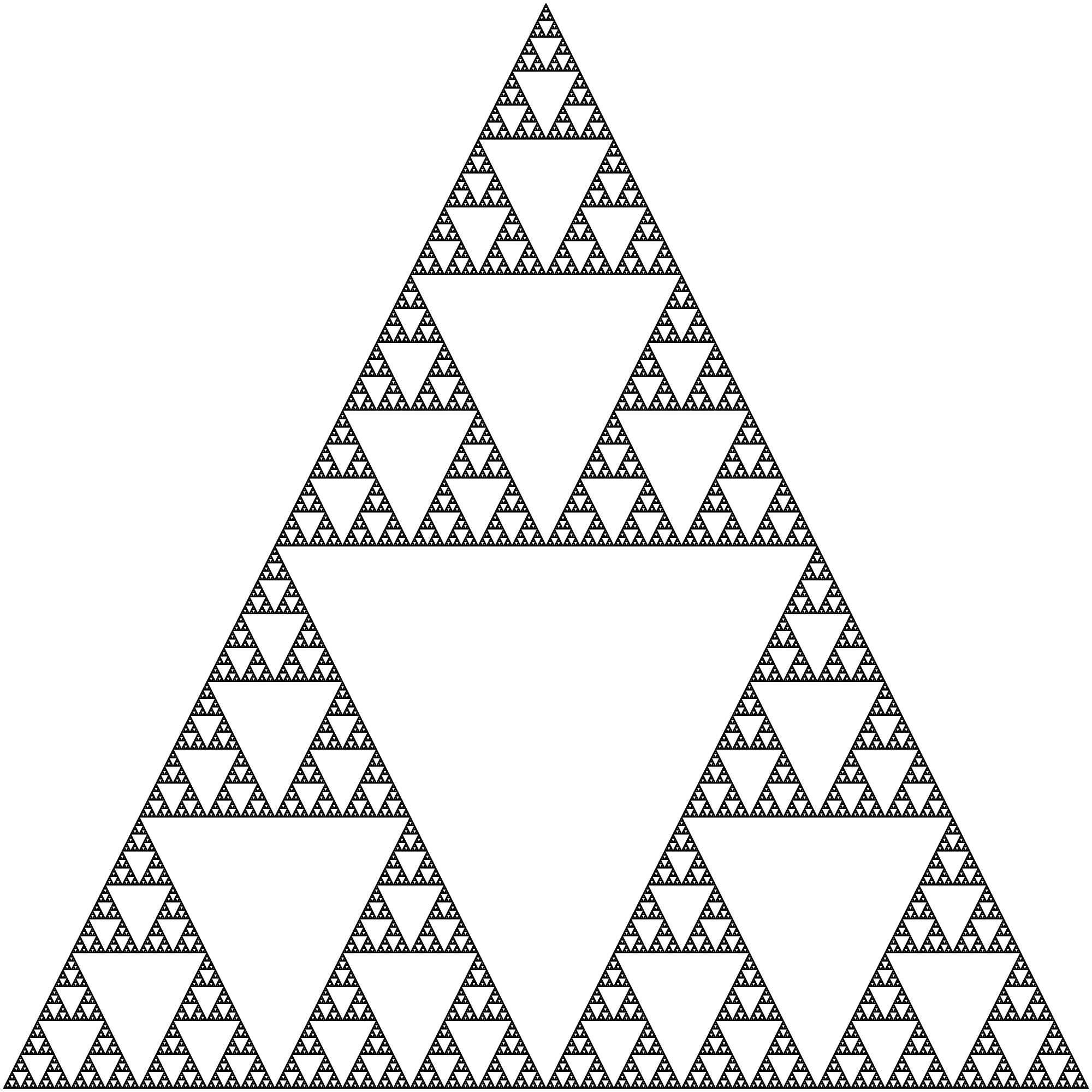} \includegraphics[width=0.48\textwidth]{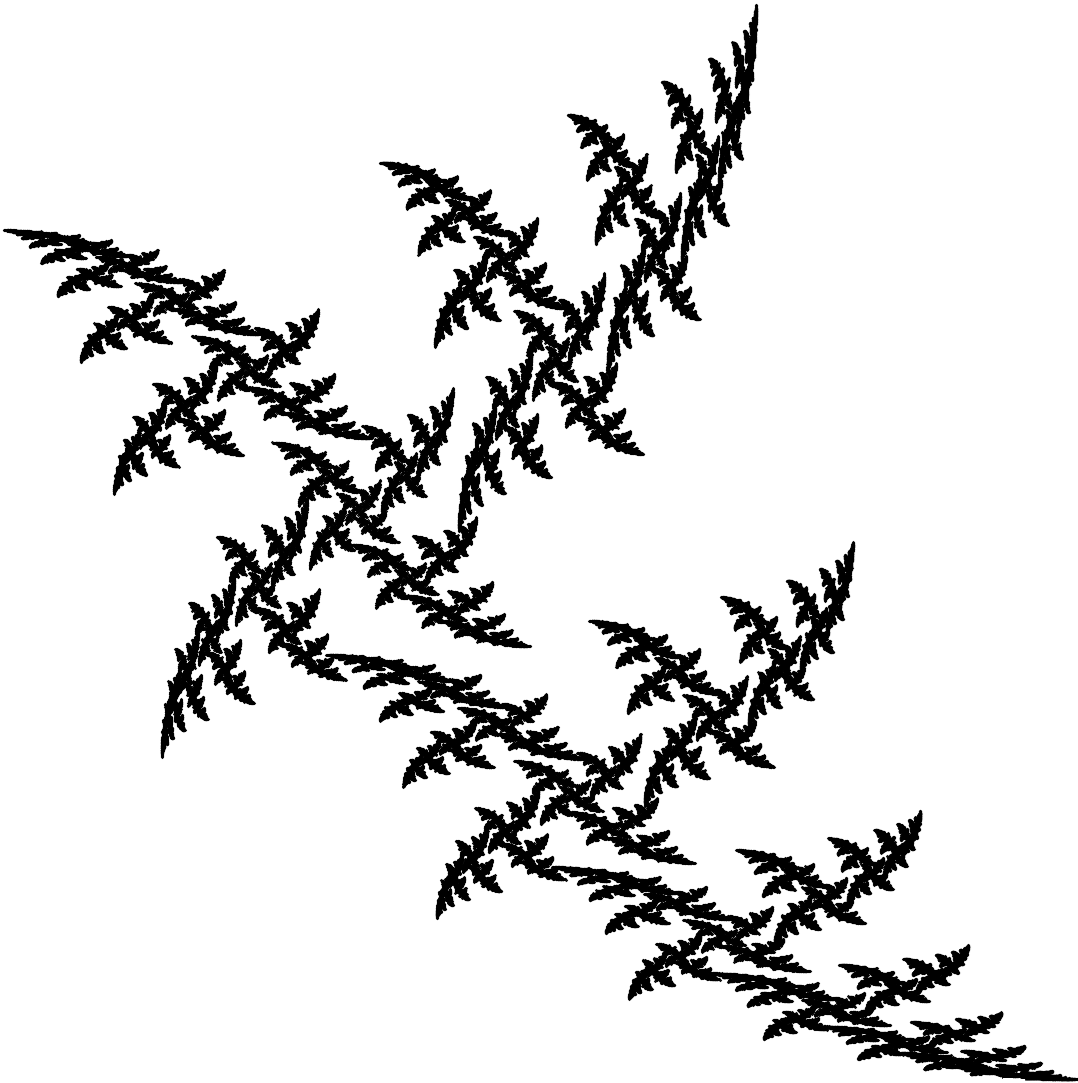}
\caption{Two homogeneous attractors of classical IFSs. On the left is the Sierpi\'nski triangle, which is a self-similar set generated by three contractions all contracting uniformly by 1/2.  On the right is a self-affine set generated by two affine contractions.}
\end{figure}

Inhomogeneous attractors are built from standard IFSs in the following way.  The ingredients are a standard IFS and a fixed  compact set $C \subseteq X$, which we call the \emph{condensation set}.  Analogous to the classical setting above, there is a unique non-empty compact set, $F_C$, satisfying
\begin{equation} \label{inhom}
F_C =  \bigcup_{i=1}^{N} S_i (F_C) \ \cup \ C,
\end{equation}
which we refer to as the \emph{inhomogeneous} attractor (with condensation $C$).   Setting $C = \emptyset$, we recover the associated  homogeneous attractor as $F_\emptyset$. When the maps in the IFS are similarities, $F_C$ is called an \emph{inhomogeneous self-similar set} and $F_\emptyset$ is called a \emph{self-similar set} with similar naming conventions for conformal maps, affine maps, etc.   Inhomogeneous attractors were introduced and studied in \cite{barndemko} (see also \cite{hata})  and are also discussed in detail in \cite{superfractals} where, among other things, Barnsley gives applications of these schemes to image compression.  Roughly speaking, the idea here is that one only needs to store the data $\{\mathbb{I}, C\}$ to recover the set $F_C$ which is much more complicated, e.g. containing infinitely many distorted and scaled copies of $C$.   

 Define the \emph{orbital set}, $\mathcal{O}$, by
\[
\mathcal{O} \ = \ C \ \cup \  \bigcup_{k \in \mathbb{N}} \ \  \bigcup_{i_1, \dots, i_k \in \{1, \dots, N\}} S_{i_1} \circ \cdots \circ S_{i_k}(C),
\]
that is, $\mathcal{O}$ is the union of the condensation set, $C$, together with all images of $C$ under compositions of maps from  the IFS.  The term orbital set was introduced in \cite{superfractals} and  this set plays an important role in the structure of inhomogeneous attractors.  Indeed,
\begin{equation} \label{structure}
F_C \ = \  F_\emptyset \cup \mathcal{O} \ = \  \overline{\mathcal{O}},
\end{equation}
where $F_\emptyset$ is the \emph{homogeneous} attractor of the IFS, $\mathbb{I}$. The formulations \eqref{structure} are straightforward to prove, see for example  \cite[Lemma 3.9]{ninaphd} in the case where  the defining maps are similarities, but  we encourage the reader to try to prove them  for themselves.  


\begin{figure}[H] \label{examples1}
	\centering
	\includegraphics[width=0.44\textwidth]{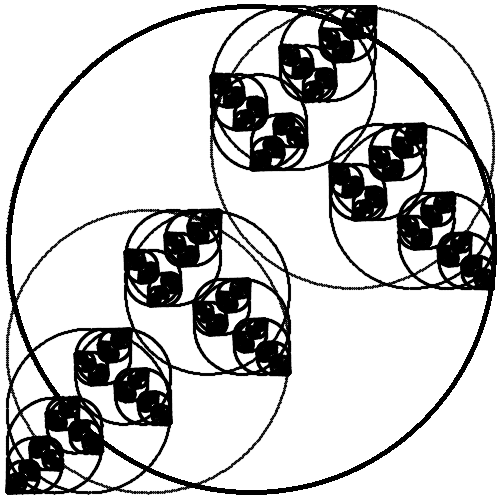}\includegraphics[width=0.44\textwidth]{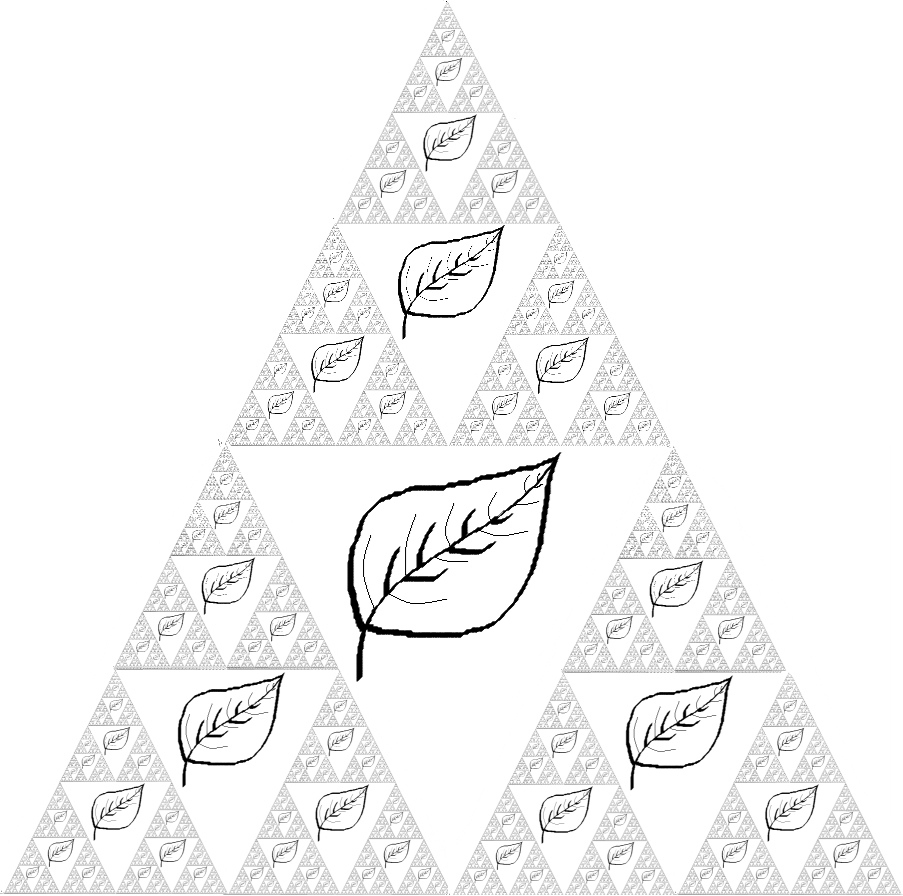}
\caption{Two inhomogeneous attractors. On the left, the IFS consists of two similarity contractions (one of which has a non-trivial rotational component).  The condensation set $C$ is the largest circle in the picture.  You can see the homogeneous attractor emerging as the circles cluster towards it.  On the right is an inhomogeneous attractor based on the Sierpi\'nski triangle.  The condensation set is the large central leaf and one can imagine this attractor depicts leaves falling to the forest floor, seen from above.}
\end{figure}

\subsection{Dimension theory of inhomogeneous attractors}

Inhomogeneous attractors are often fractal and (unlike standard IFS attractors) both the fractal behaviour of the condensation set $C$ and the underlying IFS come in to play.  One of the central concepts  in fractal geometry is that of dimension.  There are many different ways to define the  dimension of a fractal and part of the beauty of the subject is to understand how these different notions relate to each other.  For now, we work with a general dimension $\dim$.  If $\dim$  is monotone, countably stable, and stable under Lipschitz maps, then it follows immediately from (\ref{structure}) that
\begin{equation} \label{expected}
\dim F_C = \max \{ \dim  F_\emptyset, \dim C\}.
\end{equation}
See \cite{falconer} for more information about general dimensions and the properties mentioned above. In particular, \eqref{expected} holds in complete generality for the Hausdorff and packing dimensions.  Therefore, it is perhaps more interesting to consider notions of dimension which are not countably stable, such as the upper and lower box dimension or the Assouad dimension.    Even though the formula \eqref{expected} is not immediate for these dimensions, one might still expect it to hold; at least in certain situations.  However, we will see that \eqref{expected} can fail in several different ways, and exploring such examples---and identifying situations where \eqref{expected} does hold---is our main goal.  

  We  first recall the definition of the upper and lower box dimensions, discuss potential differences between them in our context, and then briefly describe some situations where other dimensions have been considered.  

 Given a bounded set $E \subseteq \rd$ and a  scale $\delta>0$, write $N_\delta(E)$ to denote the smallest number of sets of diameter $\delta$ which are needed to cover $E$.  Thus, by fixing a scale, we effectively discretize the problem of quantifying the `size' of $E$ and we can then extract a meaningful notion of dimension as the scale tends to 0.  More precisely, the \emph{upper} and \emph{lower box dimensions} of $E$ are defined by
\[
\ubd E = \limsup_{\delta \to 0} \frac{\log N_\delta(E)}{-\log \delta}
\]
and
\[
\lbd E = \liminf_{\delta \to 0} \frac{\log N_\delta(E)}{-\log \delta},
\]
respectively. If the upper and lower box dimensions agree, then we refer to the common value as the \emph{box dimension}, denoted by $\bd E$.  We refer the reader to \cite{falconer} for more details and background on the upper and lower box dimensions.

 The lower box dimension of inhomogeneous attractors  is rather more challenging to study than the upper box dimension but  was considered in some detail in \cite{fraserinhom}.  In particular, it was shown that \eqref{expected} can fail in even the simplest settings for lower box dimension.  To see why    it is  difficult to study the lower box dimension, consider a simple situation where the IFS consists of just two similarities but that the condensation set $C$ has a complicated scaling behaviour leading to $\lbd C < \ubd C$.  Then, immediately we get
\[
\max \{ \lbd  F_\emptyset, \ \lbd  C\} \leq \lbd F_C \leq \ubd F_C 
\]
and one might try to prove that the first inequality is in fact an equality (that is, try to verify \eqref{expected}).  But this is not true in general because the orbital set consists of many different copies of $C$ appearing at different scales which means that the `liminf behaviour' leading to $\lbd C$ may not be witnessed by $F_C$ and is instead `smoothened out'. Frustrated, one may then try to prove that the second inequality is an equality but, again, this is certainly not true in general.  It is true that many copies of $C$ appear at different scales and that \emph{some} of these scales may witness the `limsup behaviour' leading to $\ubd C$, but there is no reason to expect the `smoothening' to take the lower box dimension of $F_C$ all the way up to the right hand side.

An interesting connection between the study of the lower box dimension of inhomogeneous self-similar sets and the  lower box dimension of infinitely generated (homogeneous) self-similar or self-conformal sets has been identified in \cite{banaji}.  For infinitely generated attractors, the upper and lower box dimensions of the countable set of fixed points of the defining maps play a role in the dimension theory and this set of fixed points is reproduced at many scales.  This creates a similar phenomenon as we described above for  inhomogeneous attractors and leads to similar subtle features of the lower box dimension in comparison with the upper box dimension.

We can be more optimistic concerning  the \emph{upper} box dimension of inhomogeneous attractors,   as we shall see later. For definiteness, our main focus is now on verifying (or refuting)
\begin{equation} \label{boxexpected}
\ubd F_C = \max\{ \ubd F_\emptyset, \ubd C\}
\end{equation}
although we will also make some reference to \eqref{expected} for the lower box dimension.

As mentioned above, there are many other interesting dimension theoretic questions one can ask about inhomogeneous attractors.  For brevity, we only briefly summarise some of  them here, encouraging the reader to refer to the growing body of literature. The Assouad dimension of inhomogeneous attractors was studied by K\"aenm\"aki and   Lehrb\"ack \cite{kaenmaki}, where \eqref{expected} was verified for inhomogeneous self-similar sets satisfying some natural conditions.  Although \eqref{expected} is immediate for Hausdorff dimension, one might still consider the behaviour of the Hausdorff measure of $F_C$ in the Hausdorff dimension.  This problem was comprehensively studied  by Burrell \cite{burrell} for  inhomogeneous self-similar sets. A graph-directed model for inhomogeneous attractors was considered by Dubey and Verma \cite{verma} which built on work in the inhomogeneous  self-similar case from \cite{fraserinhom}.  This was further studied by B\'ar\'any--K\"aenm\"aki--Nissinen  \cite{kaenmakigd} where finer properties of the covering function were described. There is also a connection with the complex dimensions studied by Lapidus--Radunovi\'c--\v{Z}ubrini\'c  \cite{lapidus1, lapidus2}, where inhomogeneous attractors emerge as notable examples in the development of the general theory.

There is also a wealth of literature studying inhomogeneous self-similar \emph{measures}.  These are defined similarly, but with $C$ replaced by a condensation measure $\nu$ and the IFS further equipped with a (sub-)probability vector.  One may then ask dimension theoretic questions about the associated inhomogeneous measures $\mu_\nu$.   Such questions include the study of the $L^q$ dimensions, considered by \cite{liszka, olseninhom, shuqinlq},  and the    multifractal formalism, studied by Olsen and Snigireva \cite{olsenmultifractal}. In another direction, one may  study the Fourier analytic behaviour of the  inhomogeneous self-similar measures $\mu_\nu$.  This  was considered in \cite{bisbas, olsenfourier,shuqinfourier}.

We ask for forgiveness from the reader for any glaring omissions in the brief history given above, but also for omitting further discussion of the many interesting results and questions which surround inhomogeneous constructions.  From now on we focus only on the upper box dimensions of $F_C$ and consider several different settings.

\subsection{Notation}

For real-valued functions $A$ and $B$, we will write $A(x) \lesssim B(x)$ if there exists a constant $c>0$ independent of the variable $x$ such that $A(x) \leq c B(x)$, $A(x) \gtrsim B(x)$ if $B(x) \lesssim A(x)$ and $A(x) \approx B(x)$ if $A(x) \lesssim B(x)$ and $A(x) \gtrsim B(x)$.  In our setting, $x$ is normally  $\delta>0$ from the definition of box dimension or a parameter  $k \in \mathbb{N}$  related to $\delta$ by, for example, $\delta=2^{-k}$.  The implicit  constant $c$ above can depend on fixed quantities only, such as the condensation set $C$ or the IFS.

\section{Inhomogeneous self-similar sets}

In the case where the defining contractions in the IFS $\mathbb{I}$ are non-degenerate similarities, recall that  the homogeneous attractor $F_\emptyset$ is called a \emph{self-similar set} and the inhomogeneous attractor $F_C$ is called an \emph{inhomogeneous self-similar set}.  In this case, for each $i = 1, \dots, N$, there exists a well-defined contraction ratio $\text{Lip}(S_i) \in (0,1)$ such that for all $x,y \in X$
\[
|S_i(x) - S_i(y)| = \text{Lip}(S_i) |x-y|,
\]
that is, $S_i$ contracts all distances uniformly by $\text{Lip}(S_i)$.  The \emph{similarity dimension} of the self-similar set (or perhaps, more accurately, of the IFS) is defined to be the unique real solution $s$ to the Hutchinson--Moran formula
\begin{equation} \label{hutch}
\sum_{i=1}^N \text{Lip}(S_i)^s = 1.
\end{equation}
The similarity dimension is in some sense the `best guess' for the dimension of the self-similar set  $F_\emptyset$.  In fact if the `pieces' $S_i(F_\emptyset)$ do not overlap too much (for example, they are pairwise disjoint), then 
\[
\hd F_\emptyset = \lbd F =  \ubd F_\emptyset = s.
\]
It is important to note that this holds much more generally, even when the pieces overlap in non-trivial ways.  On the other hand, it does not always hold, for example if there are repeated maps in the defining IFS or `exact overlaps' which emerge later in the construction.

The following theorem was obtained in \cite{fraserinhom}.  We give the full proof, which will hopefully serve as an instructive introduction to these type of arguments.

\begin{thm} \label{main}
Let $F_C$ be an inhomogeneous self-similar set with compact condensation set $C$ and similarity dimension $s$.  Then
\[
\max \{ \overline{\dim}_\text{\emph{B}} F_\emptyset,  \overline{\dim}_\text{\emph{B}} C\} \ \leq \ \overline{\dim}_\text{\emph{B}} F_C  \ \leq \ \max \{ s,  \overline{\dim}_\text{\emph{B}} C\}.
\]
\end{thm}

\begin{proof}
The idea is to decompose the orbital set  into `small pieces' and `big pieces' (relative to a fixed covering scale $\delta>0$.  The big pieces must be covered individually and the key observation is that covering a scaled down copy of $C$ by a factor $r$ at scale $\delta$ is the same as covering $C$ at the inflated scale $\delta/r$.  On the other hand, the (infinitely many)  small pieces can be grouped together with one group for every piece which is roughly of size $\delta$.  Both of these covering strategies connect the covering number of the orbital set  to sums of contraction ratios and the covering number of $C$ itself.

Fix an IFS $\mathbb{I}=\{ S_{1}, \dots,  S_{N}\}$, where each $S_i$ is a similarity map, and  a compact condensation set  $C$.  Write $\mathcal{I} = \{1, \dots, N\}$ and $L_{\min} = \min_{i \in \mathcal{I}} \text{Lip}(S_i)$.  Let
\[
\mathcal{I}^* = \bigcup_{k \in \mathbb{N} } \mathcal{I}^k
\]
denote the set of all finite words over $\mathcal{I}$.  For $\textbf{i}  = (i_1, \dots, i_k) \in \mathcal{I}^*$, write $S_\textbf{i} = S_{i_1} \circ \cdots \circ S_{i_k}$ for the composition of maps associated to the finite word $\textbf{i}$.  Thus the `pieces' of the orbital set are the sets $S_{\textbf{i}}(C)$.  Further, write  $\textbf{i}_-  = (i_1, \dots, i_{k-1})$ and  $\lvert \textbf{i} \rvert = k$ to denote the length of the word $\textbf{i}$.  For $\delta \in (0,1]$, define a $\delta$-stopping, $\mathcal{I}(\delta)$, by
\[
\mathcal{I}(\delta) = \big\{ \textbf{i} \in \mathcal{I}^* : \text{Lip}(S_\textbf{i}) < \delta \leq \text{Lip}(S_{ \textbf{i}_-}) \big\},
\]
where we assume for convenience that $ \text{Lip}(S_{\omega}) = 1$, where $\omega$ is the empty word.  It is easy to see that, for all $\delta \in (0,1]$,
\begin{equation}\label{stopsize}
\delta^{-s} \ \leq \  \lvert  \mathcal{I}(\delta) \rvert  \ \leq \ L_{\min}^{-s} \, \delta^{-s}.
\end{equation}
Indeed, repeated application of Hutchinson's formula (\ref{hutch}) gives
\[
\sum_{\textbf{i} \in \mathcal{I}(\delta)} \text{Lip}(S_\textbf{i})^{s}=1
\]
from which we deduce
\[
1 = \sum_{\textbf{i} \in \mathcal{I}(\delta)} \text{Lip}(S_\textbf{i})^{s} \geq \sum_{\textbf{i} \in \mathcal{I}(\delta)} (\delta \, L_{\min})^s = \lvert \mathcal{I}(\delta) \rvert \,  (\delta \, L_{\min})^s 
\]
and
\[
1 = \sum_{\textbf{i} \in \mathcal{I}(\delta)} \textup{Lip}(S_\textbf{i})^{s} \leq \sum_{\textbf{i} \in \mathcal{I}(\delta)} \delta^s = \lvert \mathcal{I}(\delta) \rvert \,  \delta^s.
\]
Moreover, for all $t > s$,
\begin{equation}\label{sumsize}
\sum_{\textbf{i} \in \mathcal{I}^*} \text{Lip}(S_\textbf{i})^{t} = \sum_{k=1}^\infty \sum_{\textbf{i} \in \mathcal{I}^k} \text{Lip}(S_\textbf{i})^{t} = \sum_{k=1}^\infty \Bigg( \sum_{i \in \mathcal{I}} \text{Lip}(S_i)^{t} \Bigg)^k < \infty,
\end{equation}
by \eqref{hutch}.

By monotonicity of upper box dimension,  $\max \{ \overline{\dim}_\text{B} F_\emptyset,  \overline{\dim}_\text{B} C\} \ \leq \ \overline{\dim}_\text{B} F_C$ and so we prove the other inequality.  Since upper box dimension is finitely stable, it suffices to show that
\[
\overline{\dim}_\text{B} \mathcal{O} \leq \max\{s,  \overline{\dim}_\text{B} C\}.
\]
Let $t>\max\{s,  \overline{\dim}_\text{B} C\}$.  It follows from the definition of upper box dimension that 
\begin{equation} \label{boundy}
N_\delta(C) \lesssim \delta^{-t}
\end{equation}
for all $\delta \in (0,1]$, (with the implicit constant depending on $t$ but not $\delta$).  Also note that since $X$ is compact, the smallest number of balls of radius 1 required to cover $X$ is a finite constant $N_1(X)$.  

We are now ready to complete the proof.  For  $\delta \in (0,1]$, 
\begin{eqnarray*}
N_\delta(\mathcal{O}) &=& N_\delta \Bigg(  C \cup \bigcup_{\textbf{i} \in \mathcal{I}^*} S_{\textbf{i}}(C) \Bigg) \\ 
 &\leq& \sum_{\substack{\textbf{i} \in \mathcal{I}^*: \\ \\
\delta \leq \text{Lip}(S_\textbf{i})}} N_\delta \big( S_{\textbf{i}}(C) \big) \ +  \   N_\delta  \Bigg( \  \bigcup_{\substack{\textbf{i} \in \mathcal{I}^*: \\ \\
\delta> \text{Lip}(S_\textbf{i})}} S_{\textbf{i}}(C) \ \Bigg) \ + \ N_\delta(C)\\ 
&\leq& \sum_{\substack{\textbf{i} \in \mathcal{I}^*: \\ \\
\delta \leq \text{Lip}(S_\textbf{i})}} N_{\delta / \text{Lip}(S_\textbf{i})} (C) \ +  \   N_\delta  \Bigg( \  \bigcup_{\textbf{i} \in \mathcal{I}(\delta)} S_{\textbf{i}}(X) \ \Bigg) \ + \ N_\delta(C)\\ 
&\lesssim& \sum_{\substack{\textbf{i} \in \mathcal{I}^*: \\ \\
\delta \leq \text{Lip}(S_\textbf{i})}}\big(\delta / \text{Lip}(S_\textbf{i})\big)^{-t} \ +  \   \sum_{\textbf{i} \in \mathcal{I}(\delta)} N_{\delta / \text{Lip}(S_\textbf{i})}(X) \ + \  \delta^{-t} \qquad \text{by (\ref{boundy})}\\ 
&\leq&  \delta^{-t}  \sum_{\substack{\textbf{i} \in \mathcal{I}^*: \\ \\
\delta \leq \text{Lip}(S_\textbf{i})}} \text{Lip}(S_\textbf{i})^{t} \ +  \   N_1(X) \, \lvert  \mathcal{I}(\delta) \rvert \ + \  \delta^{-t}  \\ 
&\lesssim& \delta^{-t}  \, \sum_{\textbf{i} \in \mathcal{I}^*} \text{Lip}(S_\textbf{i})^{t} \ +  \   \delta^{-s} \ + \  \delta^{-t} \qquad \quad  \text{by  \eqref{stopsize}}\\ 
&\lesssim& \delta^{-t}
\end{eqnarray*}
by  \eqref{sumsize}.  This proves the upper bound and the theorem.
\end{proof}

In \cite[Corollary 2.6]{olseninhom} and \cite[Theorem 3.10 (2)]{ninaphd} it was proved that if each of the $S_i$ are similarities, and the sets $S_1 (F_C), \dots, S_N(F_C),  C$ are pairwise disjoint, then
\[
\overline{\dim}_\text{B} F_C = \max \{ \overline{\dim}_\text{B}  F_\emptyset,  \overline{\dim}_\text{B} C\}.
\]
This follows from Theorem \ref{main} above and the fact that, under this  strong separation condition,  
\begin{equation} \label{hutchygood}
\ubd F_\emptyset = s.
\end{equation}
However, \eqref{hutchygood} holds far more generally than this.  For example, the rather weaker \emph{open set condition} will suffice (see \cite{falconer}) but, moreover, the celebrated `exact overlaps conjecture' asserts that \eqref{hutchygood} holds for self-similar sets in the line provided $s \leq 1$ and there are no exact overlaps in the system; that is, the semigroup generated by the defining IFS is free, see \cite{peressolomyak}.  Whilst still open, this conjecture is known to hold  provided the contractions are algebraic and it also holds outside of a very small exceptional set for suitably parametrised IFSs.  We omit further discussion of this, but refer the reader to the extensive and impressive recent literature on this problem, including but not limited to \cite{hochman, rapaport}.  We state one simple corollary of Theorem \ref{main} in this direction, which combines Theorem \ref{main} with the main result in \cite{rapaport}.

\begin{cor} \label{maincor}
Let $F_C$ be an inhomogeneous self-similar set in $\mathbb{R}$.   Suppose the linear parts of the maps in the IFS are defined with algebraic parameters and that the semigroup generated by the IFS is free. Then \eqref{boxexpected} holds, that is, 
\[
\ubd F_C = \max \{ \overline{\dim}_\text{\emph{B}} F_\emptyset,  \overline{\dim}_\text{\emph{B}} C\}.
\]
\end{cor}

\begin{figure}[H] \label{examples1}
	\centering
	\includegraphics[width=\textwidth]{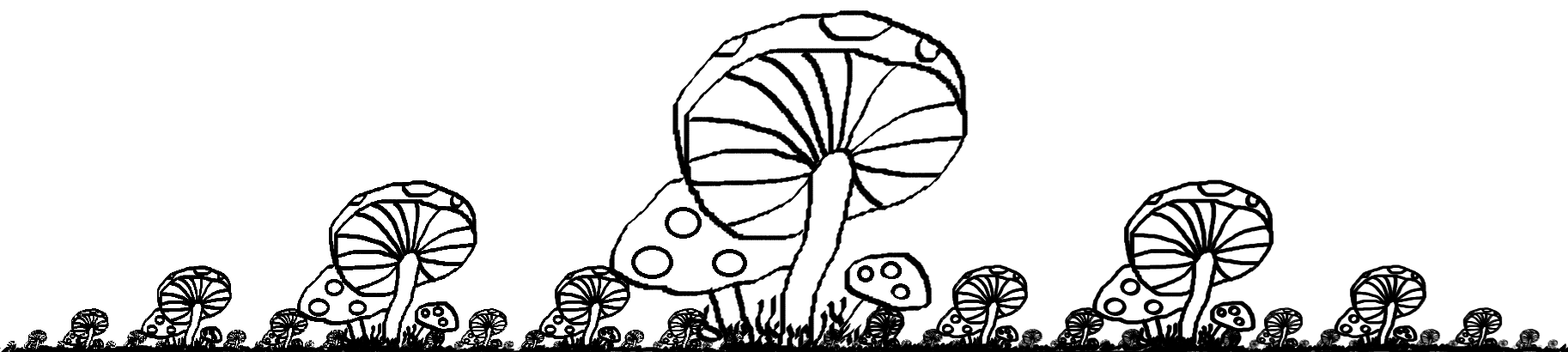}
\caption{Inhomogeneous self-similar mushroom patch. }
\end{figure}

\section{More general attractors}

Theorem \ref{main} provides a roadmap to study more general inhomogeneous attractors where there is a `best guess' for the dimension of the \emph{homogeneous} attractor.  For example, a similar result was proved in \cite{burrellaffine} for general self-affine sets with the $s$ given by the `affinity dimension' and a similar result was proved in \cite{burrell} for general self-conformal sets with the $s$ given by Bowen's formula; the unique zero of the associated topological pressure.  Similar to Theorem \ref{main} and Corollary \ref{maincor}, these results allow precise formulae to be deduced provided the `best guess' actually gives the upper box dimension of the homogeneous attractor. This is known to hold in many situations and we briefly recall some results of this type below.

For self-affine sets there is again a wealth of recent literature, building on earlier work of Falconer \cite{affine}, which says that the dimension of a self-affine set is given by the affinity dimension under only very mild assumptions; see \cite{affine2, solomyak, barany}.  However, different to the self-similar case, there are important classes of exceptional set where the dimension is smaller than the affinity dimension due to, for example, excessive alignment of cylinders; see Bedford--McMullen carpets \cite{bedford, mcmullen}.  A detailed study of inhomogeneous Bedford--McMullen carpets was provided by \cite{fraseraffine} and general families of sets where \eqref{boxexpected} fails were described.  We provide an explicit example later in Section \ref{carpetsection}.

For self-conformal sets, the appropriately formulated open set condition will suffice for the dimension of the homogeneous attractor to be given by the zero of the topological pressure; thus verifying \eqref{boxexpected}.  Such examples include inhomogeneous `cookie-cutters'.

Burrell \cite{burrell} also gave  bounds of the form provided by Theorem \ref{main} for general IFSs. In this case, the $s$ was given by the `upper Lipschitz dimension'.  In general this is a poor upper bound for the upper box dimension of the homogeneous attractor, but more can be said if the system satisfies some bounded distortion type estimates.

\section{Inhomogeneous self-similar sets with overlaps: number theoretic counterexamples} \label{resultssection}

Given the optimism provided by Theorem \ref{main} and Corollary \ref{maincor}, one might start to believe that  \eqref{boxexpected} \emph{always} holds for inhomogeneous self-similar sets.  In fact, it does not.  This was first proved in \cite{baker} and two distinct types of counterexample were given, both of which utilised the ambient space in a crucial way. The first, which we exhibit below, uses an IFS which is `trapped' in    a line but a condensation set which `releases' it into the plane.  The second is an example in $\mathbb{R}^3$ where one utilises the algebraic properties of $SO(3)$, specifically the existence of subgroups with the spectral gap property. For this family of (counter)examples, both the condensation set $C$ and the homogeneous attractor were  singletons, but the (lower) box dimension of the inhomogeneous self-similar set could be made arbitrarily close to $2$.  This latter construction is not possible in $\mathbb{R}^2$.  Given that neither  of these counterexamples   work in the line, this leaves what is our favourite open problem on inhomogeneous attractors; first stated as \cite[Conjecture 5.3]{baker}.

\begin{ques}
Consider an inhomogeneous self-similar set $F_C \subseteq \mathbb{R}$.  Is it necessarily  true that (\ref{boxexpected}) holds, that is,  
\[
\ubd F_C = \max\{ \ubd F_\emptyset, \ubd C\} ?
\]
\end{ques}

We now describe the first family of counterexamples provided by \cite{baker}.  The construction is based on Bernoulli convolutions.  Fix $\lambda \in (0,1)$, let $X = [0,1]^2$ and let $S_1, S_2 : X \to X$ be defined by
\[
S_1(x) = \lambda x \ \ \text{ and } \ \ S_2(x) = \lambda x + (1-\lambda,0).
\]
To the homogeneous IFS $\{S_1, S_2\}$, associate the condensation set
\[
C = \{0\}\times [0,1]
\]
and observe that $F_\emptyset = [0,1] \times \{0\}$ and so $\dim_\text{B} F_\emptyset = \dim_\text{B} C = 1$.  We will denote the inhomogeneous attractor of this system by $F_C^\lambda$ to emphasise the dependence on $\lambda$. 

In order to effectively release the dimension of the IFS into the plane, we need to use number theoretic properties of $\lambda$.  In particular, it must be chosen so that the orbital set spreads out the images of $C$ as quickly and uniformly as possible.  This is most dramatically achieved by   a well known class of algebraic integers known as \emph{Garsia numbers}. Recall that a Garsia number is  a positive real algebraic integer with norm $\pm 2$, whose conjugates are all of modulus strictly greater than $1.$ Examples of Garsia numbers include $\sqrt[n]{2}$  for (integers $n \geq 2$) and $1.76929\ldots$; the appropriate root of $x^{3}-2x-2=0.$ In \cite{Gar} Garsia proved  that if $\lambda$ is the reciprocal of a Garsia number, then  the associated Bernoulli convolution is absolutely continuous with a bounded density.  This `smoothness' property of the Bernoulli convolution  is another manifestation of the the `spreading out power' of  Garsia numbers. In the other direction, \emph{Pisot numbers} are the polar opposite of Garsia numbers. They lead to excessive piling up of copies of $C$ and singular Bernoulli convolutions.

\begin{figure}[H] \label{examples}
	\centering
	\includegraphics[width=0.9\textwidth]{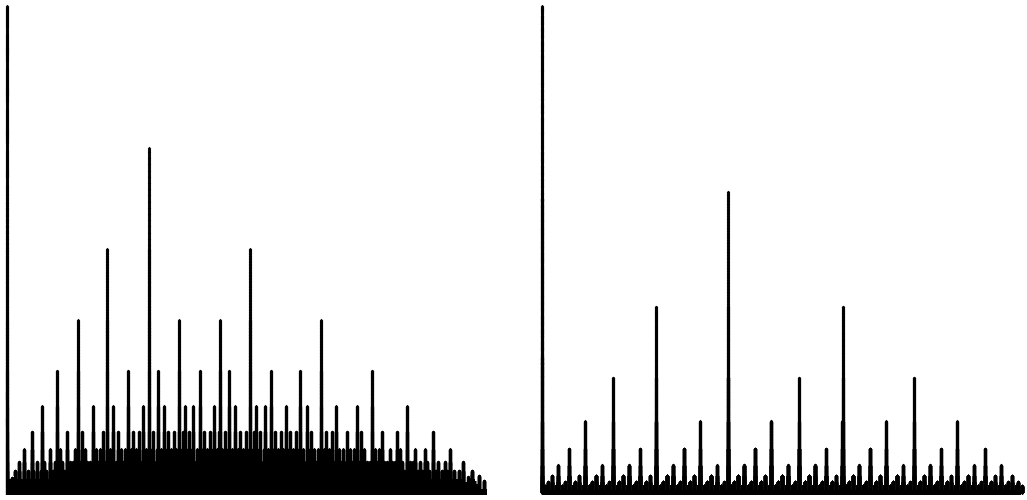}
\caption{Two  plots of $F_C^\lambda$ where $\lambda$ is chosen to be the reciprocal of $\sqrt{2}$, which is a Garsia number (left); and the reciprocal of the golden mean, which is Pisot (right). You can see that the images of $C$ are more uniformly spread out in the example on the left.}
\end{figure}

\begin{thm} \label{thmbernoulli}
If $\lambda\in(1/2,1)$ is the reciprocal of a Garsia number, then
\[
\dim_\text{\emph{B}} F_C^\lambda \ = \  \frac{\log( 4\lambda)}{\log2}  \ > \ 1.
\]
\end{thm}

 For every $\lambda\in(1/2,1)$ which is the reciprocal of a Garsia number, the set $F_C^\lambda$ thus provides a counterexample to (\ref{boxexpected}) for the upper (and lower) box dimension.  Furthermore,  this example  shows that $\overline{\dim}_\text{B} F_C^\lambda$ does not  just depend on the sets $F_\emptyset$ and $C$, but also depends on the IFS itself.  To see this, simply observe that $F_\emptyset$ and $C$ do not depend on $\lambda$, but $\overline{\dim}_\text{B} F_C^\lambda$ does.

Before we get to the proof we state a useful separation property that holds for the reciprocals of Garsia numbers and demonstrate the relevance to our situation via (\ref{S(C) equation}) below.

\begin{lma}[Garsia \cite{Gar}]
\label{Garsia's lemma}
Let $\lambda\in(1/2,1)$ be the reciprocal of a Garsia number and $(i_{k})_{k=1}^{n},(i_{k}')_{k=1}^{n}\in\{1,2\}^{n}$ be distinct words of length $n$. Then $$\Big|(1-\lambda)\sum_{k=1}^{n}i_{k}\lambda^{k-1}-(1-\lambda)\sum_{k=1}^{n}i_{k}'\lambda^{k-1}\Big|> \frac{K}{2^{n}}$$ for some strictly positive constant $K$ that only depends on $\lambda.$
\end{lma}

Observe that, for any $\I=(i_{1},\ldots,i_{n})\in\{1,2\}^{n}$, 
\begin{equation}\label{S(C) equation}
S_{\I}(C)=\{S_{\I}(0,0)\} \times  [0,\lambda^{n}] =  \left\{ (1-\lambda)\sum_{k=1}^{n}i_k\lambda^{k-1}\right\} \times [0,\lambda^{n}].
\end{equation} Combining Lemma \ref{Garsia's lemma} with (\ref{S(C) equation}), we see that whenever $\lambda$ is the reciprocal of a Garsia number the images of $C$ will be separated by a factor $K\cdot2^{-n}$. This property is the main tool we use in the proof of Theorem \ref{thmbernoulli}.

\begin{proof}[Proof of Theorem \ref{thmbernoulli}]
 Fix $\delta>0$ and partition the unit square into horizontal strips of the form $[0,1] \times (\lambda^{k+1}, \lambda^k]$ for $k$ ranging from 0 to $k(\lambda, \delta)$, defined to be the largest integer satisfying $\lambda^{k(\lambda, \delta)+1} > \delta$.  Observe that the only part of $F_C^\lambda$ which intersect the interior of the $k$th horizontal strip is
\[
\bigcup_{l=0}^{k} \ \bigcup_{\I \in \{1,2\}^l} S_{\I}(C)
\]
which is a union of vertical lines.  Within the $k$th horizontal strip, each vertical line $S_{\I}(C)$ appearing in the above expression intersects roughly  $\lambda^k/\delta$ many  squares from the $\delta$-mesh.  However, if two of these lines are too close to each another, they may not both contribute to the total intersections for $F_C^\lambda$.  In fact, the number of lines which make a contribution to the total intersections is roughly  $N_\delta( \Lambda(k) )$, where
\[
\Lambda(k) \ = \ \bigcup_{l=0}^{k} \ \bigcup_{\I \in \{1,2\}^l} S_{\I}(0,0),
\]
that is, the number of base points of the lines intersecting the $k$th horizontal  strip which lie in different $\delta$-intervals.  This yields
\begin{eqnarray*}
N_\delta\big( F_C^\lambda \big) &\approx& \delta^{-1} \ + \ \sum_{k=0}^{k(\lambda, \delta)} \big(\lambda^k/\delta \big) \, N_\delta( \Lambda(k) )
\end{eqnarray*}
where the $\delta^{-1}$ comes from the intersections below the $k(\lambda, \delta)$th horizontal strip.  It follows from Lemma \ref{Garsia's lemma} and subsequent discussion that
\[
N_\delta( \Lambda(k) )  \ \approx \  \min\{2^k,\delta^{-1}\}
\]
which is the maximum value possible and where the `comparison constants' are independent of $\delta$ and $k$, but do depend on $\lambda$, which is fixed.  Let $k_0(\delta)$ be the largest integer satisfying $2^{k_0(\delta)}<\delta^{-1}$.  It follows that
\begin{eqnarray*}
N_\delta\big( F_C^\lambda \big) &\approx& \delta^{-1} \ + \ \sum_{k=0}^{k_0(\delta)} \big(\lambda^k/\delta \big) \, 2^k   \ + \ \sum_{k=k_0(\delta)+1}^{k(\lambda, \delta)} \big(\lambda^k/\delta \big) \, \delta^{-1} \\ \\
&=& \delta^{-1} \ + \ \delta^{-1} \, \sum_{k=0}^{k_0(\delta)}(2\lambda)^k   \ + \ \delta^{-2} \, \sum_{k=k_0(\delta)+1}^{k(\lambda, \delta)} \lambda^k \\ \\
&\approx&\delta^{-1} \ + \  \delta^{-1} \, (2\lambda)^{k_0( \delta)} \ + \ \delta^{-2}\big(\lambda^{k(\lambda,\delta)} \, - \, \lambda^{k_0(\delta)} \big) \\ \\
&\approx&\delta^{-1} \ + \  \delta^{-1} \,\delta^{-\log(2\lambda)/\log2} \ + \ \delta^{-2}\big(\delta \, - \, \delta^{-\log\lambda/\log2} \big) \\ \\
&\approx&  \delta^{-1 -\log(2\lambda)/\log2}
\end{eqnarray*}
which yields
\[
\overline{\dim}_\text{B} F_C^\lambda \ = \ \underline{\dim}_\text{B} F_C^\lambda \ = \ 1 + \log(2\lambda)/\log2 \ = \  \frac{\log( 4\lambda)}{\log2}
\]
as required.
\end{proof}

The key reason that the sets $F_C^\lambda$ provide counterexamples to  (\ref{boxexpected}) is that the set $F_\emptyset$ is trapped in a proper subspace of the plane.  The underlying IFS has potential to give rise to an attractor with dimension bigger than 1, but cannot because the attractor is forced to lie in a 1-dimensional line.  However, since the condensation set does not lie in this subspace, it `releases' some of this potential dimension.

\section{Inhomogeneous self-affine sets: fractal combs and counterexamples} \label{carpetsection}

We give a construction of an inhomogeneous Bedford--McMullen carpet, which we refer to as an \emph{inhomogeneous fractal comb} which exhibits some interesting properties; in particular, failure of \eqref{boxexpected}. These examples were first constructed in \cite{fraseraffine}. The underlying homogeneous IFS is a Bedford--McMullen construction where the unit square has been divided into 2 columns of width $1/2$, and $n > 2$ rows of height $1/n$.  The IFS is then made up of all the maps which correspond to the left hand column.  The condensation set for this construction is taken as $C = [0,1] \times \{0\}$, i.e. the base of the unit square and the \emph{homogeneous} attractor is $\{0\} \times [0,1]$, that is, the left hand side of the unit square.  In particular, $\max \{ \overline{\dim}_\text{B} F_\emptyset,  \overline{\dim}_\text{B} C \} \  = \ 1$.   The inhomogeneous attractor is termed the inhomogeneous fractal comb and is denoted by $F_C^n$, to emphasise the dependence on $n$.

\begin{thm} \label{affinethm}
For $n \geq 3$,
\[
\lbd F_C^n \ = \ \ubd  F_C^n  \ = \ 2-\log 2/ \log n \ > \ 1.
\]
\end{thm} 

For every $n \geq 3$, the set $F_C^n$ thus provides a counterexample to (\ref{boxexpected}) for the upper (and lower) box dimension.   Notable here is that the underlying IFS satisfies the open set condition; this is not possible for inhomogeneous self-similar sets by Theorem \ref{main}.  A minor modification of the construction of $F_C^n$ also permits failure of \eqref{boxexpected} under the strong separation condition.  We leave the details to the reader.

Furthermore,  this example again shows that $\overline{\dim}_\text{B} F_C^n$ does not just depend on the sets $F_\emptyset$ and $C$, but also depends on the IFS itself.  To see this, observe that $F_\emptyset$ and $C$ do not depend on $n$, but $\overline{\dim}_\text{B} F_C^n$ does.  Finally, observe that, although the inhomogeneous fractal combs are subsets of $\mathbb{R}^2$ and the expected box dimension is 1, we can find examples where the achieved box dimension is arbitrarily close to 2 by letting $n \to \infty$.  This demonstrates that, even in this simple case, there is no limit to how `badly' the relationship (\ref{boxexpected}) can fail.
\begin{figure}[H]
	\centering
	\includegraphics[width=0.9\textwidth]{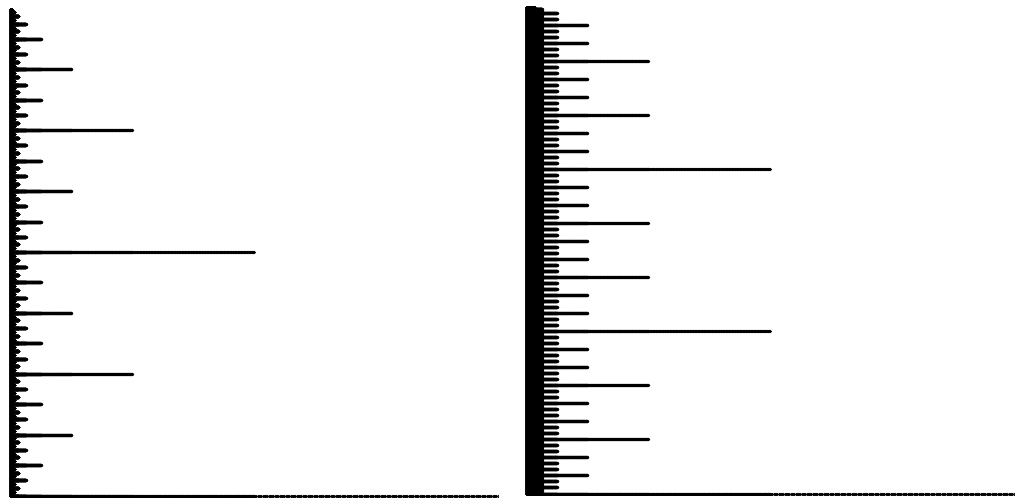}
	\caption{Two fractal combs: the inhomogeneous fractal combs $F_C^2$, with box dimension $1$ (left); and $F_C^3$, with box dimension $2-\log2/\log3>1$ (right).}
\end{figure}

Theorem \ref{affinethm} follows from a special case of the main result in \cite{fraseraffine} but we give a simple direct proof here.

\emph{Proof of Theorem \ref{affinethm}.} We again write  $\mathcal{I} = \{1, \dots, N\}$ for the set of indices describing the maps in the defining IFS and $\mathcal{I}^{k}$ for the set of words of length $k$ over $\mathcal{I}$.  In this particular case $N=n$. 

Let $\delta>0$ be a small scale and choose $m(\delta)$ to be the largest integer value of $m$ such that 
\[
\delta < n^{-m}.
\]
We choose $m(\delta)$ in this way so that the heights of the cylinders $S_\textbf{i}([0,1]^2)$ are roughly $\delta$ for $\textbf{i} \in \mathcal{I}^{m(\delta)}$.   The choice of $m(\delta)$ also means that the vertical separation of the $m(\delta)$th level images of $C$ is also roughly $\delta$. This means that  we may build an efficient cover of $F_C^n$ by treating images of $C$ up to the  $m(\delta)$th level   independently.  Covering such an image of $C$ at scale $\delta$ is simply the length of the line divided by $\delta$ and so we get
\begin{align*}
N_\delta(F_C^n)  &\approx \sum_{k=1}^{m(\delta)} \#\mathcal{I}^k\,  2^{-k}/\delta\\
&= \delta^{-1}\sum_{k=1}^{m(\delta)} (n/2)^{k}\\
&\approx \delta^{-1} (n/2)^{m(\delta)} \\
&\approx \delta^{-1} \delta^{-1+\log 2/\log n} 
\end{align*}
which proves that $\bd F_C^n = 2-\log 2/ \log n$ as required.
\hfill \qed

\section{Kleinian orbital sets: the original inhomogeneous attractors?}

Let  $n \geq 2$ be an integer and consider the Poincar\'e ball 
\[
\mathbb{D}^n = \{ z \in \mathbb{R}^n: |z| <1\}
\]
equipped with the hyperbolic metric $d $ given by
\[
|ds| = \frac{2|dz|}{1-|z|^2}.
\]
This provides a model of $n$-dimensional hyperbolic space.  The group of orientation preserving isometries of $(\mathbb{D}^n, d )$ is the group of conformal automorphisms of $\mathbb{D}^n$, which we denote by $\conp$.  A group $\Gamma \leq  \conp$ is called \emph{Kleinian} if it is a discrete   subset of $\conp$.  Kleinian groups generate fractal limit sets living on the boundary $S^{n-1}=\partial \mathbb{D}^n$ as well as  beautiful tessellations of hyperbolic space.  Both of these objects are defined via \emph{orbits} under the action of the Kleinian group.   The \emph{limit set} is defined by
\[
L(\Gamma) = \overline{\Gamma(0)} \setminus \Gamma(0)
\]
where $\Gamma(0) = \{ g(0) : g \in \Gamma\}$ is the orbit of 0 under $\Gamma$ and $\overline{\Gamma(0)}$ is the   Euclidean closure of $\Gamma(0)$.  On the other hand, hyperbolic  tessellations arise by taking the orbit of a fundamental domain for the group action.

The \emph{Poincar\'e exponent} is a coarse measure of  the rate of accumulation to the boundary.  It is defined as the exponent of convergence of the \emph{Poincar\'e series} 
\[
P_\Gamma(s) = \sum_{g \in \Gamma } \exp(-sd (0,g(0))) = \sum_{g \in \Gamma} \left(\frac{1-|g(0)|}{1+|g(0)|} \right)^s
\]
for $s \geq 0$.  That is, the  \emph{Poincar\'e exponent} is
\[
\delta(\Gamma) = \inf\{ s \geq 0 : P_\Gamma(s) <\infty\}.
\]
A Kleinian group is called \emph{non-elementary} if its limit set contains at least 3 points, in which case it is necessarily an uncountable perfect set.  In the case $n=2$, Kleinian groups are more commonly referred to as Fuchsian groups.  For more background on hyperbolic geometry and Kleinian groups see \cite{beardon, bowditch,stratmann}.

At this point, you may recognise a few familiar sounding concepts.  First, the orbit $\Gamma(0)$ is an orbital set with condensation $C=\{0\}$, albeit the orbit is under a group action rather than a semigroup action as in the IFS case, but many similarities remain. Not least that the orbit accumulates on the fractal limit set, in this case $L(\Gamma)$. Furthermore, the  Poincar\'e exponent shares many features with the similarity dimension, recall \eqref{hutch}, or perhaps more accurately  with the zero of the topological pressure associated to a conformal system.  In particular, it provides a `best guess' for the dimensions of the limit set.

In \cite{bartlett},  Kleinian orbital sets were formally introduced and studied in the context of inhomogeneous attractors.   Fix a non-empty  set $C \subseteq \mathbb{D}^n$ and a Kleinian group $\Gamma$.    The \emph{orbital set} is defined to be
\[
\Gamma(C) = \bigcup_{g \in \Gamma} g(C).
\]
It is easy to see that the  limit set is  contained in the Euclidean closure of any orbital set.  

There is a celebrated connection between hyperbolic geometry (especially Fuchsian groups) and the artwork of M.~C.~Escher. Orbital sets fall very naturally into this discussion since many of the memorable images from Escher's work  are orbital sets (rather than tessellations).  Here  $C$ could be a  large central bat or fish, which is then repeated many times on smaller and smaller scales towards the boundary of $\mathbb{D}^n$. In some sense this makes Kleinian orbital sets the original inhomogeneous attractors since they already appeared  in the work of Escher.

\begin{figure}[H]  
	\centering
	\includegraphics[width=0.45\textwidth]{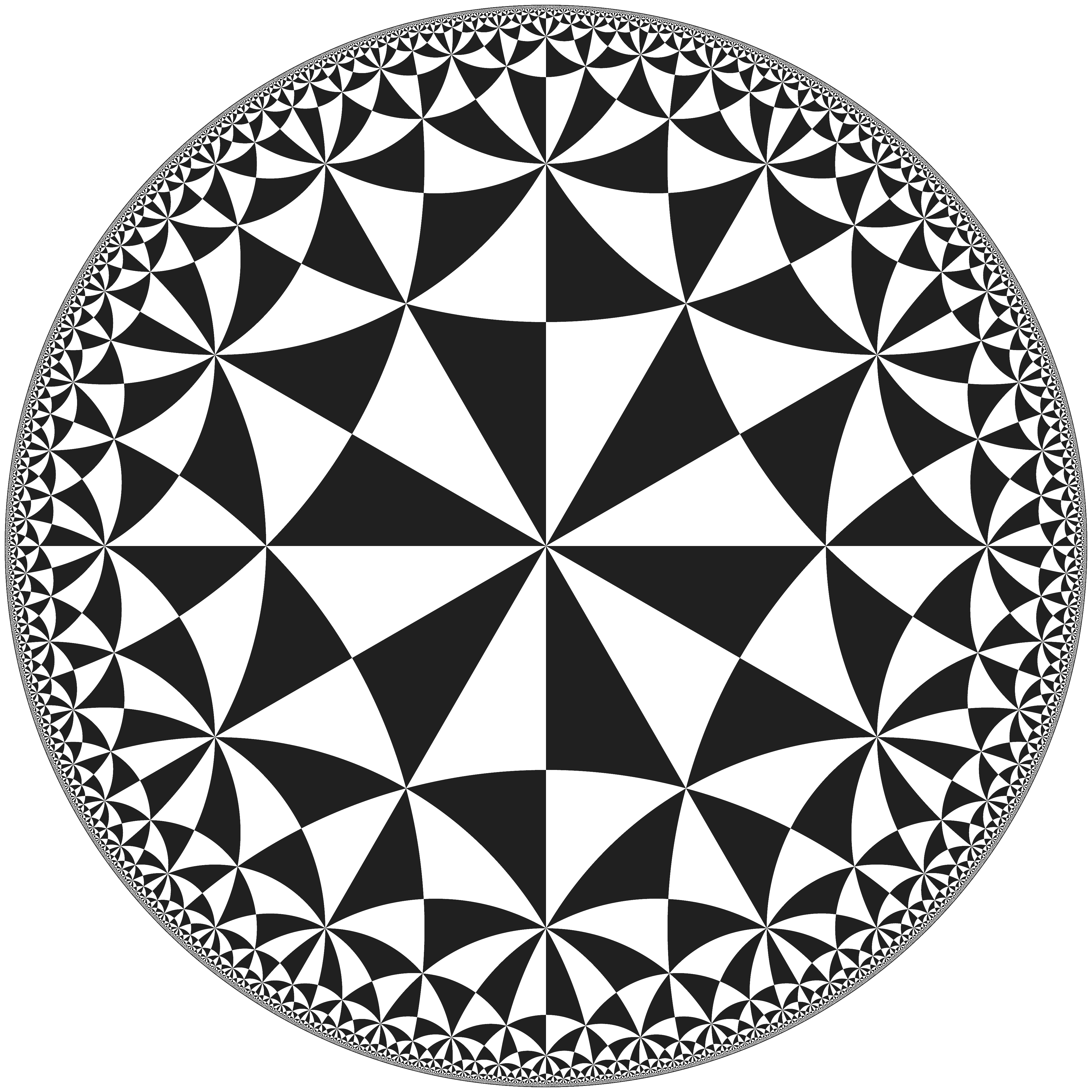}
\caption{Inhomogeneous tiling: a tessellation of the hyperbolic plane by a triangular fundamental domain. This is also an orbital set with $C$ being any one of the triangles.  The triangles in the tessellation are coloured such that no two triangles sharing an edge are coloured the same.}
\end{figure}

The main result of \cite{bartlett}  was a complete characterisation of the upper box dimension of Kleinian orbital sets with $C$ bounded in the hyperbolic metric. Boundedness in the hyperbolic metric simply means that $C$ is uniformly bounded away from the boundary $S^{n-1}$.  For clarity, we emphasise  that we compute the dimension of the orbital set with respect to the \emph{Euclidean} metric on $\mathbb{R}^n$.

\begin{thm}
\label{thm:2}
Let $\Gamma$ be a Kleinian group acting on $\mathbb{D}^n$ and $C$ be a non-empty  subset of $\mathbb{D}^n$ which is bounded in the hyperbolic metric. Then 
$$
\ubd \Gamma(C) = \max\left\{\ubd L(\Gamma), \ubd C, \delta(\Gamma)\right\}.
$$
\end{thm}

This should be compared with Theorem \ref{main}.  It is perhaps noteworthy that no assumptions are imposed on the Kleinian group and that each of the three terms appearing on the right hand side is necessary; see \cite{bartlett} for more information.  If we impose the additional assumption that $\Gamma$ is non-elementary and geometrically finite, then $\ubd L(\Gamma) = \delta(\Gamma)$   (see \cite{bishopjones,su}) and the result reduces to something more reminiscent of \eqref{boxexpected}.

In the proof of the above theorem,  essential use was made of the assumption that $C$ is bounded (in the hyperbolic metric).  It turns out that this is necessary and we exhibit a family of counterexamples from \cite{bartlett}  below.  

\begin{thm}
There exists a Kleinian (in fact, Fuchsian) group $\Gamma$ acting on $\mathbb{D}^2$ and a condensation set $C \subseteq \mathbb{D}^2$ such that
\[
\ubd L(\Gamma) =  \ubd C =  \delta(\Gamma) = 0
\]
but
\[
\lbd \Gamma(C) = \ubd \Gamma(C) = 1.
\]
\end{thm}

\begin{proof}
The set $C$ and group $\Gamma$ are very simple.  The work is in proving that the orbital set has large dimension, and this relies on some number theory.  Let $\alpha>1$ and $\beta \in (0,1)$ be such that $\log \alpha$ and $\log \beta$ are rationally independent, that is, $\log \alpha /\log \beta \notin \mathbb{Q}$.  Here and throughout $\log$ is the natural logarithm.  For example, $\alpha = 2$, and $\beta = 1/3$ suffices.  Let
\[
C=\{ 1-\beta^n : n \in \mathbb{N}\} \subseteq \mathbb{D}^2
\]
noting that $C$ is unbounded in $(\mathbb{D}^2, d)$.  Let $h \in \textup{con}^+(\mathbb{D}^2)$ be the hyperbolic element with repelling fixed point $-1$ and attracting fixed point $1$ given by
\[
h(z) = \frac{(\alpha+1)z + (\alpha-1)}{(\alpha-1)z + (\alpha+1)}.
\]
Let $\Gamma = \langle h \rangle$ be the elementary Fuchsian group generated by $h$.  It is an elementary but instructive exercise to show that
\[
\ubd L(\Gamma) =  \ubd C =  \delta(\Gamma) = 0.
\]
We encourage the reader to try this for themselves. In order to prove that $\lbd \Gamma(C) = \ubd \Gamma(C)= 1$ we show that $\Gamma(C)$ is dense in $(-1,1) \subseteq \mathbb{D}^2$, recalling that the box dimensions are stable under taking closure. This shows that the   box dimensions are at least 1, but the orbital set is contained in $(-1,1)$ and so they are also at most 1.  The orbital set has a straightforward description due to the simplicity of $\Gamma$ and $C$.  Indeed, 
\begin{align*}
\Gamma(C) &= \{ h^m( 1-\beta^n) : m \in \mathbb{Z}, \, n \in \mathbb{N}\} \\ 
&= \left\{ \frac{(\alpha^m+1)( 1-\beta^n) + (\alpha^m-1)}{(\alpha^m-1)( 1-\beta^n) + (\alpha^m+1)} : m \in \mathbb{Z}, \, n \in \mathbb{N}\right\} \\ 
&= \left\{ \frac{ 2-\alpha^m\beta^n -\beta^n}{2+\alpha^m\beta^n -\beta^n} : m \in \mathbb{Z}, \, n \in \mathbb{N}\right\}
\end{align*}
noting that we switch the role of $m$ and $-m$ in the final expression, which is fine since $m \in \mathbb{Z}$. Let $y \in (0,\infty)$ be such that $\log y \in \mathbb{Q}$. Since $\log \alpha/\log \beta \notin \mathbb{Q}$ we can find sequences $m_k \in \mathbb{Z}, \, n_k \in \mathbb{N}$ such that
\[
\alpha^{m_k}\beta^{n_k} \to y
\]
as $k \to \infty$.  This is a standard application of Dirichlet's approximation theorem.  Moreover, again since  $\log \alpha$ and $\log \beta$ are rationally independent, we necessarily have $n_k \to \infty$ as $k \to \infty$.  Therefore
\[
\frac{ 2-\alpha^{m_k}\beta^{n_k} -\beta^{n_k}}{2+\alpha^{m_k}\beta^{n_k} -\beta^{n_k}} \to \frac{2-y}{2+y}
\]
as $k \to \infty$.  The set
\[
\left\{ \frac{2-y}{2+y} : y \in (0,\infty) \text{ and } \log y \in \mathbb{Q}\right\}
\]
is dense in $(-1,1)$ and the density of $\Gamma(C)$ in $(-1,1)$ follows.
\end{proof}

\section*{Acknowledgements}

I am grateful to Lars Olsen for first introducing me to inhomogeneous attractors. I was quite inspired by Nina Snigireva's PhD thesis on the topic, back when I first read it in 2010 or so. I am also grateful to Simon Baker,  Amlan Banaji, Tom  Bartlett,   Stuart Burrell, Antti K\"aenm\"aki,  Andras Math\'e, and Alex Rutar for various helpful discussions over the years.

I wrote this expository article in May and June 2024 during pleasant trips  to Madison, Banff and Warwick.  I blame jet-lag or dodgy airport Wi-Fi for any typos or inconsistencies!

\end{document}